\documentclass[11pt]{amsart}

\usepackage{amssymb}
\usepackage{enumitem}
\usepackage{booktabs, multirow}
\usepackage{array}
\usepackage{url}
\usepackage[utf8]{inputenc}
\usepackage{mathtools, hyperref}

\makeatletter
\newcommand{\labeltarget}[1]{\Hy@raisedlink{\hypertarget{#1}{}}}
\makeatother

\newcommand{\PreserveBackslash}[1]{\let\temp=\\#1\let\\=\temp}
\newcolumntype{C}[1]{>{\PreserveBackslash\centering}p{#1}}

\newtheorem{thm}{Theorem}
\newtheorem{cor}[thm]{Corollary}
\newtheorem{lem}[thm]{Lemma}
\newtheorem{prop}[thm]{Proposition}

\theoremstyle{definition}

\newtheorem{exam}[thm]{Example}
\newtheorem{rem}[thm]{Remark}

\newcommand{\R}{\mathbb{R}}
\newcommand{\C}{\mathbb{C}}
\newcommand{\D}{\mathbb{D}}
\newcommand{\DA}{\mathcal{A}}

\newcommand{\norm}[1]{\left\Vert#1\right\Vert}
\newcommand{\abs}[1]{\left\vert#1\right\vert}
\newcommand*\CLOSED[1]{\overline{#1}}
\newcommand*\CONJ[1]{\overline{#1}}
\newcommand*{\EVALX}{\delta}

\newcommand*{\HOLO}{\mathcal H}

\newcommand*\SUPP{supp}

\newcommand*\BOP{\mathcal L}  

\DeclareMathOperator{\EN}{wem}

\newcommand*\APPRE{\mathcal{AE}} 

\let\originalleft\left
\let\originalright\right
\renewcommand{\left}{\mathopen{}\mathclose\bgroup\originalleft}
\renewcommand{\right}{\aftergroup\egroup\originalright}

\allowdisplaybreaks

\begin{document}

\title[Weak null maximum and integration of families of multiplication operators]{Weak null maximum and integration of families of multiplication operators}

\author[ ]{David Norrbo}

\email{david.norrbo@ntnu.no}
\address{Faculty of Science and Engineering, Åbo Akademi University, 20500 Åbo, Finland}
\address{Department of Mathematics and Statistics, School of Mathematical and Physical Sciences, University of Reading, Whiteknights, PO Box 220, Reading RG6 6AX, UK.}



\begin{abstract}
Let $X$ be a reflexive Hardy space or weighted Bergman space on the unit disk in the complex plane. For a bounded linear operator $S$ on $X$, let $\EN(S):= \sup_{(f_n)} \limsup_n \|Sf_n\|$, that is, the supremum of cluster points of $n\mapsto \|S f_n\|$, where $(f_n)$ is any unit norm weakly null sequence. This quantity coincides with the essential norm on the reflexive weighted Bergman spaces. 

For a suitable family $\{ g_t : t\in]0,1[ \}$ of bounded analytic functions on the unit disk, we characterize when one can exchange $\EN(\cdot)$ and integration over $t$ of the multiplication operators $M_{g_t}$, also called analytic Toeplitz operators. In other words, we characterize when  $\EN(  \int M_{g_t}\, dt  ) =   \int \EN( M_{g_t} ) \, dt $; if the functions $g_t,t\in]0,1[$ can be continuously extended to the unit circle, we obtain a neat function-theoretic characterization.
\end{abstract}

\maketitle

\section{Introduction}

Given a Banach space \(X\), we denote the weak null sequences of elements in a set \(M\subset X\) by \(W_0(M)\). If \(B_X\) is the unit ball with boundary, \(\partial B_X\), and \(S\in \BOP(X)\) (bounded operator), we examine the quantity
\[
\EN(S) := \sup_{(f_n)\in W_0(\partial B_X)} \limsup_n \norm{S f_n}.
\]
If \(X\) is reflexive, then \(W_0(\partial B_X) \neq \emptyset\) since \(B_X\) is sequentially weakly compact, see for example \cite[3.5 Reflexive Spaces]{Brezis-2011}. Notice that although the weak closure of \(\partial B_X\) contains \(0\) for a general Banach space, there is no sequence converging weakly to \(0\) if \(X\) has the Schur property (weak and norm convergence coincide for sequences), for example if \(X =\ell^1\), in which case \(\EN(\cdot)\) should be expressed differently. The main results in this work only concern particular reflexive spaces. The quantity \(\EN(S)\) is always smaller than or equal to the essential norm, \(\norm{S}_e \) (distance to compact operators), and it coincides with the essential norm if \(X\) is a Hilbert space or if \(X\) belongs to a certain class of Banach spaces, including the weighted Bergman spaces \(A^p_\alpha, \, p>1\). This conclusion for the weighted Bergman spaces is made via a property called \((M_p)\) in which case the formula for the essential norm is given in \cite[p.~499]{Werner-1992}. The fact that \(A^p_\alpha\) has the property can be seen in \cite[Corollary 3.6]{Kalton-1995} and the three lines preceding the corollary. For more details we refer the reader to \cite{Arxiv:ML-2025} and \cite{Lindstrom-2024C}.

In particular, given a family of bounded operators \(\{ S_t : t\in]0,1[ \}\), one can ask when 
\begin{equation}\label{eq:generalStatement}
\EN\Big(     \int_0^1  S_t \, dt  \Big)  =  \int_0^1 \EN(S_t) \, dt.
\end{equation}
For the quantity on the left to be well-defined, we assume that the family of operators satisfies \(W(X)\), that is, \(t\mapsto S_t f\) is continuous for every \(f\in X\) and \(\int_0^1 \norm{S_t f}_X \, dt < \infty \); see for example \cite[pp.~295--298]{Katznelson-2004}. In this case, the left-hand side of \eqref{eq:generalStatement} will never exceed the right-hand side, which in general is not true, see for example \cite[bottom of page 5]{Arxiv:EssIntWCO-2025}.

In this work, we examine when \eqref{eq:generalStatement} holds, and we focus on \(S_t\) being a multiplication operator for every \(t\in]0,1[\) and \(X\) being a separable reflexive space of analytic functions on the unit disk, either a Hardy space \(H^p\) or a weighted Bergman space \(A^p_{\alpha}, \, p>1\). In \cite{Arxiv:EssIntWCO-2025}, the author deals with the same question concerning families of certain weighted composition operators, which do not include the multiplication operators. In addition to expanding the knowledge about when \eqref{eq:generalStatement} holds, this article serves as a less technical introduction to the study of \eqref{eq:generalStatement} in general, compared to \cite{Arxiv:EssIntWCO-2025}. In contrary to the non-sharp sufficient and necessary conditions obtained in \cite{Arxiv:EssIntWCO-2025}, we are in the setting of families of multiplication operators able to obtain characterizations. 

There are two main results. Theorem \ref{thm:aCharacterization} gives a characterization of when \eqref{eq:generalStatement} holds in the most general setting of this work. Assuming the multiplication symbols are continuous on the boundary, Theorem \ref{thm:mainDA} yields a function-theoretic characterization of \eqref{eq:generalStatement}. In addition, Proposition \ref{prop:functionTheoreticWX} gives a function-theoretic characterization of \(W(X)\) when \(X=H^p\) or \(X=A^p_{\alpha} \), and  \(p\in[1,\infty[\). As a corollary to the main theorems, we remark that the weak null maximum in \eqref{eq:generalStatement} can be exchanged with the norm or the essential norm, see Corollary \ref{cor:normIneq} and a few lines preceding it for more details. 

The linear space of analytic functions on the complex unit disk is denoted by \(\HOLO(\D)\). The rotation-invariant unit measure on the unit circle \(\partial \D\) is denoted by \(m\) and the translation-invariant area measure in \(\C\) yielding \(\D\) to have unit size is denoted by \(A\). In polar coordinates, \(z=re^{iu}\), the measures have the form \(dm(z) = du/2\pi\) and \(dA(z) = r\, dr\, du / \pi\). For \(p\in]0,\infty[\) and \(\alpha>-1\), the Hardy spaces and the weighted Bergman spaces are defined as
\[
H^p := \bigg\{ f\in \HOLO(\D) :  \norm{f}_{H^p} := \lim_{r\to 1 }\Big(  \int_{\partial \D} \abs{ f(rz) }^p  \, dm(z) \Big)^{\frac{1}{p}}      \bigg\}
\]
and
\[
A^p_\alpha :=\bigg\{ f\in \HOLO(\D) :  \norm{f}_{A^p_{\alpha}} := \Big(  \int_{\D} \abs{ f(z) }^p  \, dA_\alpha(z) \Big)^{\frac{1}{p}}      \bigg\},
\]
respectively, where $dA_\alpha(z) = (1+\alpha)(1-\abs{z}^2)^{\alpha} \, dA(z)$. The spaces are reflexive Banach spaces when \(p>1\). Functions \(f\in H^p,\, p>0\), have nontangential limits at almost every point on the boundary \(\partial\D\), see for example \cite[Chapter 2]{Duren-1970}. We will adopt the notation of \(f\in H^p\) to be a function defined everywhere on \(\CLOSED{\D}\); on \(\D\) it is analytic, on \(\partial \D\) it is defined as its nontangential limit, whenever it exists and zero otherwise. For \(f\in H^p, \, p>0\), its restriction \(f|_{\partial \D} \in L^p(\partial \D)\), and the norm can be expressed as \((\int_{\partial \D} \abs{f}^p \, dm)^{1/p}\), see for example \cite[Theorem 2.6]{Duren-1970}). The linear space of bounded analytic functions, we denote by \(H^\infty\), which is a Banach algebra when equipped with the norm \( f\mapsto \norm{f}_\infty:= \sup_{z\in\D} \abs{f(z)}\). For information about these spaces, we refer the reader to the monographs \cite{Zhu-2005,Zhu-2007,Duren-1970} and \cite{Hoffman-1962}. The subspace of \(H^\infty\) consisting of continuous functions on the closed unit disk, \(C(\CLOSED \D)\), is denoted by \( \DA \).

One important tool when dealing with multiplication operators on Hardy spaces or weighted Bergman spaces is the \emph{approximate evaluation maps}, which are sequences of elements of unit norm \( (\phi_n) \subset  \partial B_X\) for which there exists a \(\xi\in\partial \D\) such that \(\lim_n \sup_{z\in\D\setminus B(\xi,\epsilon)} \abs{\phi_n(z)} = 0\) for every \(\epsilon>0\). For a given \(\xi\in\partial \D\), the family of such sequences is denoted by \(\APPRE(\xi)\).

\section{Maximizing sequence for multiplication operators on \(H^p\) and \(A^p_\alpha\) }

Next, we present some classical results concerning the evaluation maps. Let \(p\in]0,\infty[\) and \(\alpha>-1\). The norm of the evaluation maps, \(\delta_z \colon f \mapsto f(z), \, z\in \D\), on \(H^p\) and \(A^p_\alpha\) are given by
\begin{equation}\label{eq:NormOfEval}
\norm{\delta_z}_{(H^p)^*} = (1-\abs{z}^2)^ { -\frac{1}{p}} \quad \text{ and } \quad \norm{\delta_z}_{(A^p_\alpha)^*} = (1-\abs{z}^2)^ { -\frac{2+\alpha}{p}}, 
\end{equation}
respectively. Moreover, for any sequence \((z_n)\subset \D\)---we will specify one later---the norm of \(\delta_{z_n}\) on the respective space is attained by
\begin{equation}\label{eq:ExtremalFunctionsForEval}
f^{H^p}_n\colon w\mapsto  \frac{  (1-\abs{z_n}^2)^{\frac{1}{p}}   }{ (1-\CONJ{z_n}w)^{\frac{2}{p}}  }  \quad \text{ and }   \quad    f^{A^p_\alpha}_n\colon w\mapsto  \frac{ (1 - \abs{z_n}^2)^{  \frac{2+\alpha}{ p }   }  }{  (1-\CONJ{z_n}w)^{2\frac{2+\alpha}{p} }}. 
\end{equation}
The fact that these functions have unit norm is easily seen by a change of variables; for the weighted Bergman spaces, see for example, \cite[Proposition 4.3]{Zhu-2007}. The fact that the given quantity in \eqref{eq:NormOfEval} is an upper bound can be found in, for example, \cite[Theorem 9.1 \& 4.14]{Zhu-2007}. 

Let \(X=H^p\) or \(X=A^p_\alpha\). Due to the lattice order on the underlying \(L^p\)-spaces, \(\abs{f_1(z)} \leq \abs{f_2(z)}\ \forall z\in \D \ \Rightarrow \norm{f_1}_X \leq \norm{f_2}_X\), we have \(\norm{M_g} \leq \norm{g}_\infty\), and since the evaluation functionals are bounded, we have
\[
\norm{g}_\infty = \sup_{z\in \D}  \abs{g(z)} \sup_{f\in B_X}  \frac{ \abs{\EVALX_z (f)}  }{ \norm{\EVALX_z} }  = \sup_{f\in B_X} \sup_{z\in \D}   \abs{\frac{\EVALX_z (M_g f) }{\norm{\EVALX_z}} }  \leq   \sup_{f\in B_X} \norm{M_g f}_X  = \norm{M_g}, 
\]
and we can conclude the inequalities above are equalities. In particular, since \(\CLOSED{\D}\) is compact, there is \(z_0\in \CLOSED{\D}\) and \((z_n)\subset \D\) such that \(z_n \to z_0\) and
\[
\lim_n \sup_{f\in B_X}  \abs{\frac{\EVALX_{z_n} (M_g f) }{\norm{\EVALX_{z_n}}} }  = \norm{M_g}  = \norm{g}_\infty.
\]
By the maximum modulus principle \(z_0\in\partial \D\). Moreover,
\begin{equation}\label{eq:MaxSeq}
\norm{g}_\infty \geq \norm{M_g f^X_n} \geq \frac{ \abs{ \EVALX_{z_n}(f^X_n g) } }{  \norm{\EVALX_{z_n}} } = \abs{g(z_n)} \stackrel{ n \to \infty }{\longrightarrow}  \norm{g}_\infty.
\end{equation}
We have now obtained the following lemma:

\begin{lem}\label{lem:MaximizingApproxEvalForMultOp}
 Let \(p\in]0,\infty[\), \(\alpha>-1\), and let \(X=H^p\) or \(X=A^p_{\alpha}\). For \(g\in H^\infty\) , there is a convergent sequence \( (z_n)_n\subset \D\) such that the induced sequence \( (f^X_n)_n \) is an approximate evaluation map that maximizes \(M_g\colon X \to X\). In other words, \(\lim_n \norm{ M_g f^X_n    } = \norm{M_g} = \norm{g}_\infty\).
\end{lem}

It is worth mentioning the following classical result on reflexive Banach spaces of analytic functions with bounded evaluation functionals connecting the weak topology with the function-theoretic geometry on \(\D\). If the space \(X\) is reflexive, the unit ball \(B_X\) is weakly compact. In particular, the boundedness of evaluation maps together with Banach-Steinhaus theorem yields that weak convergence coincides with uniform convergence on compact subsets of \(\D\), illuminating that an approximate evaluation map is a special case of weakly null sequences. 

Another important result is that \(B_X\) equipped with the weak topology is metrizable if and only if the dual \(X^*\) is separable, see for example \cite[Theorem 3.29]{Brezis-2011}. In particular since the polynomials are dense in  \(H^p\) and \(A^p_{\alpha}\), \(p\in]0,\infty[\), \(\alpha>-1\) (see \cite[Proposition 2.6 and Corollary 4.26]{Zhu-2005}), if \(p>1\), the weak topology coincides with uniform convergence on compact subsets and restricted to \(B_X\) the weak topology is metrizable.

\section{Some tools and facts}

We begin the section with a diagonalization argument:

\begin{rem}[Diagonalization argument for weakly null sequences] \label{rem:diagArg}
If for \(k\in \mathbb N\), \((x_n^{(k)})\subset B_X\) is a weak null sequence, then we can create a new sequence \((y_n)\) by choosing \(1<n_1<n_2<\ldots\) and \(n_1',n_2',\ldots \geq 1\), and defining
\[
y_n := \begin{cases}  
x_n^{(1)} &  n=1,\ldots,n_1-1; \\
x_{n+n_1'}^{(2)} &  n=n_1,\ldots,n_2-1; \\
x_{n+n_2'}^{(3)} &  n=n_2,\ldots,n_3-1; \\
\ldots & \ldots.
\end{cases}
\]
If there exists a metric \(d_w\) inducing the weak topology on \(B_X\), then by choosing the shifts \(n_j', j=1,2,\ldots\) large enough we can grant that \(d_w(y_n,0)< \frac{1}{n_j}\) whenever \(n>n_j\) for \(j=1,2,\ldots\). This wouldn't be possible if we could not measure the distance to zero with respect to the weak topology. The process is executed by choosing the parameters in the following order: \(n_1,n_1',n_2,n_2',\ldots\), where \(n_1,n_2,\ldots\) corresponds to the usual diagonalization argument while \(n_1',n_2',\ldots\) are chosen to ensure \((y_n)\) remains weakly null.
\end{rem}

We also mention that by the Banach-Steinhaus theorem \(t\mapsto S_t, t\in]0,1[\) is pointwise continuous implies \(\{S_t: t\in[\epsilon,1-\epsilon]\}\) is uniformly bounded for every \(\epsilon>0\). Using this, we obtain the following proposition.

\begin{prop}\label{prop:functionTheoreticWX}
Let \(X=H^p\) or \(X=A^p_{\alpha} \), where  \(p\in[1,\infty[\) and \(\alpha>-1\), and let \(g_t\in H^\infty, t\in]0,1[\). The family \(\{M_{g_t}:t\in]0,1[\}\) satisfies \(W(X)\) if and only if 
\begin{enumerate}[label = {(\alph*)}]
\item  for every \(t_0\in]0,1[\), it holds that  \( \lim_{t\to t_0} \norm{g_t-g_{t_0} }_X = 0 \);  \label{eq:Wcond1}
\item for every compact set \(K\subset ]0,1[\), it holds that \(\sup_{t\subset K} \norm{g_t}_\infty<\infty\); and  \label{eq:Wcond2}
\item \(\int_0^1 \norm{g_t}_\infty \, dt <\infty \).  \label{eq:Wcond3}
\end{enumerate}
\end{prop}
\begin{proof}
The relation between \ref{eq:Wcond3} and \(W(X)\) follows from \(\norm{ M_g }=\norm{g}_\infty, g\in H^\infty\). Assuming \(t\mapsto M_{g_t}\) is pointwise continuous, an application of the Banach-Steinhaus theorem yields that \ref{eq:Wcond2} holds. It is clear that if \(\{ M_{g_t} : t\in]0,1[ \}\) satisfies \(W(X)\), \(t\mapsto M_{g_t}1 = g_t\) is continuous, and hence, \ref{eq:Wcond1} holds. 

Conversely, assume \ref{eq:Wcond1} and \ref{eq:Wcond2} hold. For any polynomial \(p\in X\), 
\[ 
 \limsup_{t\to t_0}  \norm{(g_t-g_{t_0})p }_X \leq  \norm{p}_\infty \limsup_{t\to t_0}  \norm{g_t-g_{t_0} }_X = 0.
\]
Hence, for every \(\epsilon>0\) and compact set \(K\subset ]0,1[\) containing \(t_0\) in its interior, the denseness of polynomials ensures that there exists a polynomial \(p\) such that for \(t\) close enough to  \(t_0\)
\begin{align*}
\norm{ (M_{g_t} - M_{g_{t_0}} ) f }_X &\leq \norm{ (M_{g_t} - M_{g_{t_0}} ) (f-p) }_X + \norm{ (M_{g_t} - M_{g_{t_0}} ) p }_X \\
& \leq  2\sup_{t\in K}\norm{g_t}_\infty  \norm{ f-p }_X + \norm{p}_\infty \norm{ g_t - g_{t_0}   }_X <\epsilon.
\end{align*}

\end{proof}

\begin{rem}
If \(X=A^p_{\alpha} \), where  \(p\in[1,\infty[\) and \(\alpha>-1\), then the norm-convergence in \ref{eq:Wcond1}  can be replaced by uniform convergence on compact subsets \(\D\).
\end{rem}

\begin{rem}[A multiplication operator] \label{rem:intOpIsMult}
Let \(\{g_t : t\in]0,1[\}\subset H^\infty\) and assume the induced family of multiplication operators satisfies \(W(X)\), where \(X=H^p\) or \(X=A^p_\alpha\), \(p\in]0,\infty[\) and \(\alpha>-1\). The integration operator will have the form \(  f\mapsto \int_0^1  M_{g_t} f \, dt = \int_0^1 g_t  \, dt \, f \). It follows from Proposition \ref{prop:functionTheoreticWX} and dominated convergence  that \( \int_0^1 g_t \, dt \in C(\CLOSED{\D}) \) if \(g_t\in \DA \), \(  t\in ]0,1[\). In general, since \(t\mapsto M_{g_t}1\) is continuous and the evaluation maps are bounded, an application of the Banach-Steinhaus theorem yields \( \lim_{t\to t_0} \sup_{z\in K} \abs{g_t(z) - g_{t_0}(z) }  \) for every compact subset \(K\subset \D\). Now, Cauchy's formula gives continuity of the nth derivative, more precisely, for every \(t_0\in ]0,1[\) and compact set \(K\subset \D\),
\[
\lim_{t\to t_0} \sup_{z\in K} \abs{g_t^{(j)}(z) - g_{t_0}^{(j)}(z) } = 0  \quad  j= 1,2\ldots. 
\]
By the Leibniz theorem, we obtain that \(F_n:=\int_{\frac{1}{n}} ^{1-\frac{1}{n}}  g_t  \, dt \in \HOLO(\D)\), and so the uniform limit on compact subsets  \(\int_0 ^1  g_t  \, dt\) is also analytic. We can conclude that \(\int_0^1 M_{g_t} \, dt\) is a multiplication operator with symbol \( \int_0 ^1  g_t  \, dt \in H^\infty \). If \(g_t\in \DA\), then \( \int_0 ^1  g_t  \, dt \in \DA\).
\end{rem}

\section{Main results}

Let \(p\in]1,\infty[\), \(\alpha>-1\) and \(X=H^p\) or \(X=A^p_{\alpha} \). In this section we characterize the families \(\{M_{g_t} : t\in ]0,1[\}\) satisfying \(W(X)\) such that the inequality
\begin{equation}\label{eq:mainEquality}
\sup_{(f_n)\in W_0(\partial B_X)} \limsup_n  \norm{\int_0^1 M_{g_t}  f_n \, dt }_X \leq \int_0^1  \sup_{(f_n)\in W_0(\partial B_X)} \limsup_n \norm{ M_{g_t}  f_n  }_X \, dt
\end{equation}
is an equality. Our first characterization partitions the problem into a Minkowski's-inequality-type problem and a uniform maximizing sequence problem.

\begin{thm}\label{thm:aCharacterization}
Let \(p\in]1,\infty[\), \(\alpha>-1\), \(X=H^p\) or \(X=A^p_{\alpha} \) and \(\{M_{g_t} : t\in ]0,1[\}\) satisfies \(W(X)\), where \(g_t\in H^\infty\). The inequality \eqref{eq:mainEquality} is an equality if and only if there is a weakly null unit norm sequence \((\phi_n)\) such that 
\begin{equation}\labeltarget{eq:i1}\tag{I1}\label{eq:ineq1}
\lim_n  \norm{\int_0^1 M_{g_t} \phi_n \, dt }_X = \lim_n  \int_0^1 \norm{ M_{g_t} \phi_n }_X \, dt, 
\end{equation}
and
\begin{equation}\labeltarget{eq:i2}\tag{I2}\label{eq:ineq2}
 \liminf_n \norm{ M_{g_t} \phi_n }_X  =  \sup_{(f_n)\in W_0(\partial B_X)} \limsup_n \norm{ M_{g_t} f_n }_X
\end{equation}
for almost every \(t\in]0,1[\). Moreover, \((\phi_n)\) can be chosen to be an approximate evaluation map. 
\end{thm}
\begin{proof}
Assume first that \eqref{eq:ineq1} and \eqref{eq:ineq2} hold. From \eqref{eq:ineq2} it follows that \(  \lim_n \norm{ M_{g_t} \phi_n }_X \) exists, so by dominated convergence, the assumptions yield
\[
\lim_n  \norm{\int_0^1 M_{g_t} \phi_n \, dt }_X = \int_0^1  \sup_{(f_n)\in W_0(\partial B_X)} \limsup_n \norm{ M_{g_t} f_n }_X \, dt.
\]
The statement now follows from
\[
\lim_n  \norm{\int_0^1 M_{g_t} \phi_n \, dt }_X \leq    \sup_{(f_n)\in W_0(\partial B_X)} \limsup_n       \norm{ \int_0^1 M_{g_t} f_n \, dt }_X    \leq  \int_0^1  \sup_{(f_n)\in W_0(\partial B_X)} \limsup_n \norm{ M_{g_t} f_n }_X \, dt.
\]

Assume now that \eqref{eq:mainEquality} is an equality. Since \( \int_0^1 M_{g_t} \, dt\) is a multiplication operator, there is an approximate evaluation map \((\phi_n)\)  (see Lemma \ref{lem:MaximizingApproxEvalForMultOp}) such that the first equality below holds:

\begin{equation}\label{eq:6ImplyChar}
\begin{split}
\sup_{(f_n)\in W_0(\partial B_X)} \limsup_n  \norm{\int_0^1 M_{g_t} f_n \, dt }_X &= \lim_n  \norm{\int_0^1 M_{g_t} \phi_n \, dt }_X \leq \limsup_n  \int_0^1 \norm{ M_{g_t} \phi_n }_X \, dt \\
&\leq   \int_0^1 \limsup_n \norm{ M_{g_t} \phi_n }_X \, dt \\
&\leq     \int_0^1  \sup_{(f_n)\in W_0(\partial B_X)} \limsup_n \norm{ M_{g_t} f_n }_X \, dt.
\end{split}
\end{equation}
By the assumption, all quantities above are equal. Hereafter, we let \((\phi_n)\) denote a suitable subsequence of \((\phi_n)\) (which will still be an approximate evaluation map) for which the limit \(\lim  \int_0^1 \norm{ S_t \phi_n }_X \, dt\) exists. Note that the sequence \((\phi_n)\) is a weakly null sequence that satisfies \eqref{eq:ineq1}. We claim that the limit \(\lim_n  \norm{ M_{g_t} \phi_n }_X \) exists for almost every \(t\in ]0,1[\). If this was not the case, then there is a subsequence \((\phi_{n_k})\) and a set \( M\subset  ]0,1[\) of positive Lebesgue measure, such that 
\[
\limsup_k  \norm{ M_{g_t} \phi_{n_k} }_X  < \limsup_n  \norm{ M_{g_t} \phi_n }_X, 
\]
for all \(t\in M\). In view of \eqref{eq:6ImplyChar}, we have a contradiction with \eqref{eq:mainEquality} being an equality.

Since 
\[
\lim_n \norm{ M_{g_t} \phi_n }_X \leq  \sup_{(f_n)\in W_0(\partial B_X)} \limsup_n \norm{ M_{g_t} f_n }_X
\]
for almost every \(t\in]0,1[\), the assumption gives equality for almost every \(t\in]0,1[\), that is, \eqref{eq:ineq2} is satisfied. Here we used the fact that, given two real valued functions \(f_1\) and \(f_2\), if \(f_1\leq f_2\) and \(\int f_1 = \int f_2\), then \(f_1 = f_2\) almost everywhere on the path of integration.

\end{proof}

It is worth noticing that the sufficiency of the condition holds for other families of operators too. In view of multiplication operators, condition \eqref{eq:ineq2} loosely says that \(\abs{g_t(z)}\)  should be maximized as \(z\) approaches \(z_0\in\partial \D\), which should be independent of \(t\). Condition \eqref{eq:ineq1} says nothing about maximizing the operators \(M_{g_t}\), but that the argument of \(g_t(z)\) should, in the limit, be the same for almost every \(t\in]0,1[\) (modulus can be moved inside the integral without changing the value), and that Minkowski's inequality, in a certain limit sense, is an equality for \(\abs{g_t(z)}\). Next, we focus on the case \(g_t\in\DA \) and obtain a simpler characterization for  \eqref{eq:mainEquality} to be an equality.

We will now illuminate, by an application of Prokhorov's theorem, how the dynamic Minkowski's inequality (involving limits as in \eqref{eq:ineq1}) can be an equality in situations where the standard (static) Minkowski's inequality would be a strict inequality. Let \((\phi_n)\subset \partial B_X\) be a sequence converging weakly to \(\phi\in B_X\). By Prokhorov's theorem (using the fact that  \(\CLOSED{\D}\) and \(\partial \D\) are compact), there is a subsequence \((\phi_{n_k})\) and a Borel measure \(\mu = \mu^{X,(\phi_{n_k})}\) such that for every \(F\in C(\CLOSED{\D})\)
\[
\begin{array}{l l}
\lim_k   \int_{\D}  F(z)  \abs{\phi_{n_k}(z)}^p \, dA_\alpha(z)  = \int_{\CLOSED{\D}} F(z)  \, d\mu(z)  & \text{if } X=A^p_\alpha,  \\[.5cm] 
\lim_k   \int_{\partial \D}  F(z)  \abs{\phi_{n_k}(z)}^p \, dm(z)  = \int_{\partial \D} F(z)  \, d\mu(z)  & \text{if } X=H^p.
\end{array}
\]
Notice that in the case \(X=A^p_\alpha\), the support \( \SUPP(\mu)\subset \partial \D\) if \(\phi_{n_k}\) converges uniformly to zero on compact subsets of \(\D\), equivalently, \( \phi_{n_k} \xrightharpoonup{} 0\) (weakly).

If \(g_t\in\DA \), then \(\int_0^1 g_t\, dt\in\DA \), so an application of dominated convergence and Prokhorov's theorem yields \eqref{eq:ineq1} is equivalent to
\begin{equation}\label{eq:MinkowskiAfterProkhorov}
\bigg( \int_{ \partial \D}  \abs{ \int_0^1 g_t(z) \, dt }^p  \, d\mu(z) \bigg)^{\frac{1}{p}}  =  \int_0^1 \Big(  \int_{ \partial \D}   \abs{g_t(z)}^p  \, d\mu(z) \Big)^{\frac{1}{p} } \, dt  
\end{equation}
for some probability (Borel) measure \(\mu\). In general, by theory for Minkowski's inequality, see for example \cite{Hardy-1988}, the condition above holds if and only if there exist functions \(g^{(1)}\colon ]0,1[\to [0,\infty[\) and \( g^{(2)} \colon \D \to \C \), and a \(\theta\in \partial \D\) such that \( g_t(z) = \theta g^{(1)}(t) g^{(2)} (z) \) for almost all \(t\in]0,1[\) and \(z\in \SUPP(\mu)\). However, by Theorem \ref{thm:aCharacterization}, we can assume \((\phi_n)\in \APPRE(\xi)\) for some \(\xi\in\partial\D\). Thus, \(\SUPP(\mu) = \{\xi\} \), in which case \eqref{eq:MinkowskiAfterProkhorov} holds if and only if there is a \(\theta \in \partial \D\)  such that  \( \theta g_t(\xi) \in [0,\infty[ \) for almost every \(t\in]0,1[\). 

Condition \eqref{eq:ineq2} holds, when \((\phi_n)\) is an approximate evaluation map, if and only if there exists \( \xi\in\partial \D \) such that \( \abs{g_t(\xi)} = \norm{g_t}_\infty\) for almost every \(t\in]0,1[\). We have obtained:
\begin{thm}\label{thm:mainDA}
Let \(p\in]1,\infty[\), \(\alpha>-1\) and \(X=H^p\) or \(X=A^p_{\alpha} \), and let \(g_t\in\DA \) for \(t\in]0,1[\). If \(\{M_{g_t} : t\in ]0,1[\}\) satisfies \(W(X)\), the inequality
\[
\EN\Big(     \int_0^1  M_{g_t} \, dt  \Big)  \leq  \int_0^1 \EN( M_{g_t}) \, dt
\]
is an equality if and only if there exists \(\xi,\theta\in\partial \D\)  such that for almost every \(t\in]0,1[\), we have \( \theta g_t(\xi) = \norm{g_t}_\infty\). 
\end{thm}

Notice that \(\EN(S)\leq \norm{S}_e\leq \norm{S}\) for any bounded operator \(S\), and \(\EN(M_g) = \norm{M_g}_X \) for \(g\in H^\infty\), \( X=H^p\) or \(X=A^p_\alpha\) , where \(p\in]1,\infty[\), \(\alpha>-1\), see Lemma \ref{lem:MaximizingApproxEvalForMultOp}. We have, therefore, obtained the following corollary. 

\begin{cor}\label{cor:normIneq}
Let \(p\in]1,\infty[\), \(\alpha>-1\) and \(X=H^p\) or \(X=A^p_{\alpha} \), and let \( g_t\in\DA \) for \(t\in]0,1[\). If \(\{M_{g_t} : t\in ]0,1[\}\) satisfies \(W(X)\), the following are equivalent: 
\begin{itemize}
\item There exists \(\xi,\theta\in\partial \D\)  such that for almost every \(t\in]0,1[\), \( \theta g_t(\xi) = \norm{g_t}_\infty\) holds. \\[-6pt]
\item Any of the quantities \(\EN\Big(     \int_0^1  M_{g_t} \, dt  \Big), \norm{   \int_0^1  M_{g_t} \, dt  }_e,  \norm{   \int_0^1  M_{g_t} \, dt  } \) equals any of \(    \int_0^1  \EN(  M_{g_t}  ) \, dt \), \(   \int_0^1  \norm{  M_{g_t}  }_e  \, dt \) and \(   \int_0^1  \norm{   M_{g_t}  }  \, dt  \).\\[-6pt]
\item All of the quantities mention above are equal.
\end{itemize}
\end{cor}

We end this section with a simple example. Prior to that, recall that a Blaschke product is a function \(B\colon \D \to \D\) of the form
\[
B(z) = z^m  \prod_{n=1}^\infty \frac{ \CONJ{a_n} }{ \abs{a_n} } \frac{a_n - z}{ 1-\CONJ{a_n} z}, \quad  m=0,1,2,\ldots,
\]
where \((a_n)\subset \D\) satisfies \(\sum (1-\abs{a_n})<\infty\). In the example, we will use the fact that \(\abs{B(z)}\) has nontangential limit 1 at almost every point on the boundary \(\partial \D\). For information about Blaschke products, we refer the reader to \cite[p.~63-66]{Hoffman-1962} and \cite[p.~19-20]{Duren-1970}.

\begin{exam}
Let \(p>1\) and \(c\in \R\). We consider the family \( \{ M_{g_{t,c}} : t\in]0,1[ \}\) of operators acting on \(H^p\), where \(g_{t,c}(z):= (c+t+z)B(z), t\in]0,1[,  z\in \CLOSED{\D}\), where \(B\) is a Blaschke product whose zeros contain a sequence \((\zeta_n)\) converging to \(1\). The function \(g_{t,c}\) is not continuous at \(1\) due to the factor \(B\), because \(\lim_n \abs{B(\zeta_n)} = 0\) and there is a sequence \((z_n)\) approaching \(1\) tangentially such that \(\lim_n \abs{B(z_n)} = 1\). This can be seen by a diagonalization argument, because the radial limits of \(\abs{B}\) are \(1\) almost everywhere. In view of \eqref{eq:MaxSeq}, for \(c\geq 0\), the approximate evaluation map \(f_n^X\) induced by the sequence \( (z_n) \) maximizes \(\abs{g_{t,c}}, t\in]0,1[\), proving \eqref{eq:ineq2}. Moreover,  
\begin{align*}
\lim_n  \norm{\int_0^1 M_{g_{t,c}} f_n^X \, dt }_{H^p} &= \lim_n  \norm{ B \int_0^1 (c+t+\cdot) \, dt  \, f_n^X  }_{H^p}  = \int_0^1 (c+t+1) \, dt  \\
&=    \int_0^1  \lim_n  \norm{ (c+t+\cdot)   \, B \,  f_n^X   }_{H^p} \, dt  =  \lim_n  \int_0^1 \norm{ M_{g_t} f_n^X  }_{H^p} \, dt
\end{align*}
  so \eqref{eq:ineq1} holds. By Theorem \ref{thm:aCharacterization}, we can conclude that the family \( \{ M_{g_{t,c}} : t\in]0,1[ \} \) satisfies \eqref{eq:generalStatement} when \(c\geq 0\). Similarly, if \(c\leq -1\) the same conclusion can be made. If \(c\in]-1,0[\), \eqref{eq:ineq2} fails. Indeed, the maximum for \(\abs{g_{t,c}},  t > -c\) is only obtained by a sequence \((z_n)\) converging to \(1\) while a sequence converging to \(-1\) is needed to maximize \(\abs{g_{t,c}},  t < -c\). Therefore, no approximate evaluation map can make \eqref{eq:ineq2} satisfied, so by Theorem \ref{thm:aCharacterization}, there is a strict inequality instead of equality in \eqref{eq:generalStatement} for the family \( \{ M_{g_{t,c}} : t\in]0,1[ \} \) when \(c\in]-1,0[\). 

\end{exam}

\section{Further research}
The condition \(W(X)\) is not optimal, only sufficient. A natural generalization of \(W(X)\) is obtained by replacing the demand of \(t\mapsto S_t f\) being continuous for every \(f\in X\), with conditions of \(t\mapsto S_t f\) being Bochner integrable for every \(f\in X\) and develop theory from there.

One could also develop theory for other types of families \(\{S_t :t \in ]0,1[\}\) than the multiplication operators and weighted composition operators. Letting \(X\) be other spaces than Hardy and Bergman spaces is also to my knowledge unknown territory.

\bibliographystyle{amsplain}
\bibliography{bibliography}

\end{document}